\documentclass[a4paper,12pt]{amsart}
\usepackage{amssymb}

\newtheorem{theorem}{Theorem}[section]
\newtheorem{definition}{Definition}[section]

\newtheorem{proposition}{Proposition}[section]
\newtheorem{lemma}{Lemma}[section]
\newtheorem{remark}{Remark}[section]

\newcommand{\R}{\mathbb R}
\newcommand{\N}{\mathbb N}
\newcommand{\Z}{\mathbb Z}

\newcommand{\C}{\mathbb C}

\newcommand{\ds}{\displaystyle}

\newcommand{\dx}{\, \mathrm{d}x}
\newcommand{\mt}{\mathcal{T}}
\newcommand{\mcc}{\mathcal{C}}
\newcommand{\dsigma}{\, \mathrm{d}\sigma}

\author[J.~Benameur$~^{a}$, Ch.~Elhechmi$~^{b}$ and G. Benhenda$~^{c}$ ]{Jamel Benameur, Chokri Elhechmi and Gmar Benhenda}

\address{$a$: Department of Mathematics College of Sciences,
King Saud University, Riyadh 11451, Kingdom of Saudi Arabia}
\email{\sl jbenameur@ksu.edu.sa}
\address{$b,c$:\;Research Laboratory Mathematics and Applications LR17ES11, Department of Mathematics, Faculty of Science of Gab\`es, university of Gab\`es; Tunisia}
\email{\sl chokri.elhechmi@issatm.rnu.tn}
\email{\sl benhendagmar@gmail.com}
\title[Laplace problem with an exponential nonlinear ...]
{Laplace problem with an exponential nonlinear boundary condition}
\begin{document}
	\begin{abstract}
In this paper, we establish a new result for the Laplace problem with  exponential Robin boundary conditions posed on the unit disk in $\R^2$. More precisely, we prove the existence and uniqueness of a solution under suitable smallness assumptions on the boundary data. Our approach relies on an iterative method combined with periodic Sobolev embedding results.
	\end{abstract}


	\maketitle
	\tableofcontents

	\section{Introduction}
 In this paper,  we study an elliptic problem posed  on the open unit disc $\Omega=D(0,1) \subset \R^2$, with an  exponential boundary condition prescribed on a part of the boundary: 
$$
(S)\left\{\begin{array}{lll}
\displaystyle\Delta u=  0,& \;{\rm in }&\Omega,\\
\displaystyle \,u=   0, &\;{\rm on}\;&\Gamma_D,\\
\displaystyle\frac{\partial u}{\partial n}=  \phi(x) &\;{\rm on }&\;\Gamma_N,\\
\vspace{3mm}
\displaystyle\frac{\partial u}{\partial n}+\varphi(x)  \ds  \left({e^{\alpha u}- e^{-(1-\alpha )u}}\right)  =  g(x) &\;{\rm on}&\;\Gamma_R,
\end{array}\right.
$$
Here, the boundary $\Gamma = \partial \Omega = S(0,1)$ is decomposed into three open subsets $\Gamma_D$, $\Gamma_N$, and $\Gamma_R$, modulo a negligible set, satisfying

$$(H1)\left\{\begin{array}{l}
\Gamma=\overline{\Gamma_D}\cup \overline{\Gamma_N}\cup\overline{\Gamma_R},\\
\Gamma_D\cap \Gamma_N=\Gamma_D\cap\Gamma_R=\Gamma_N\cap\Gamma_R=\emptyset,\\
\sigma(\Gamma_D)>0,\;\sigma(\Gamma_N)>0,\;\sigma(\Gamma_R)>0,
\end{array}\right.
$$
with $\sigma $ is the  superficial measure on $\partial\Omega$.\\
The data of the problem satisfy 
$$\phi \in L^2(\Gamma_N) ,\quad  \phi \not \equiv 0, \quad g \in L^2(
\Gamma_R),\quad\varphi\in C^0(\overline{\Gamma_R}) \text{ and }\alpha \in (0,1).$$
{ For $\alpha \in (0,1)$, we introduce  the function $f_\alpha$ defined by}
$$
f_\alpha(r) = \left\{
\begin{array}{llll}
\ds \frac{e^{\alpha x}- e^{-(1-\alpha )r}}{r} & \mbox{ if } r \neq 0,\\
1 & \mbox{ if } r = 0,
\end{array}
\right.
$$
then, problem $(S)$ can be  rewritten as  
$$
(S)\left\{\begin{array}{lll}
\displaystyle\Delta u = 0, &\;{\rm in}&\;\Omega,\\
\displaystyle u = 0, &\;{\rm on}&\;\Gamma_D,\\
\displaystyle\frac{\partial u}{\partial n} = \phi(x) & \;{\rm on }&\;\Gamma_N,\\
\displaystyle\frac{\partial u}{\partial n}+\varphi(x) f_\alpha(u)u = g(x)  &\;{\rm on}&\;\Gamma_R.
\end{array}\right.
$$
\begin{remark}
The function $f_\alpha$ satisfies the following properties:
\begin{itemize}
\item[{\bf(P1)}] $\forall\, r \in \R:\;\;f_\alpha(r) > 0$.
\item[{\bf(P2)}] $\forall\, r \in \R:\;\;0\leq f_\alpha(r)r^2 \leq 2\ds\sum_{k=1}^\infty \frac{\delta^k}{k!} |r|^{k+1}$, \text{ where } $\delta = \max(\alpha, 1-\alpha)$.
\item[{\bf(P3)}] $\forall \,r \in \R\setminus\{0\}:\;\; \leq f_\alpha(r) \leq 2\delta(1 +\delta|r|e^{\delta|r|})$.
\end{itemize}
 The proof of properties {\bf(P1)} and {\bf(P2)} are obtained by direct  calculation, while the proof of {\bf(P3)} is given in  {\rm{\bf Appendix A}}.
\end{remark}
%
\noindent A first approximation of problem $(S)$ is obtained by replacing the nonlinear boundary condition on $\Gamma_R$ with the linear Robin condition 
 $$
\ds\frac{\partial u}{\partial n}+\varphi(x) u    =  g(x)\;{\rm on}\;\Gamma_R,
$$
 This linear case has been studied by several authors. In particular, S. Chaabane {\it et al.} (\cite{chok-matcom,CJ1}) established existence and uniqueness results by applying the Lax-Milgram theorem to the variational formulation associated with the problem. In addition, numerical results were provided for both the direct problem and the inverse problem of identifying the Robin coefficient.\\
In the case where the boundary condition on  $\Gamma_R$ is quadratic, specifically:
$$
\displaystyle\frac{\partial u}{\partial n}+\varphi(x)  u +\psi(x)u^2 =  g(x) \;{\rm on}\;\Gamma_R.
$$
J. Benameur {\it et al.} \cite{J-C} established the existence of solutions in an appropriate Hilbert space $V$. Uniqueness was proved  within a closed ball of $V$. These results were obtained under suitable assumptions on the functions $\varphi$ and $\psi$.  The proofs rely on an iterative method based on the construction of a sequence of linear boundary value problems, 
which converges to the solution  of the original problem with the quadratic nonlinear boundary condition.
\\ 
A generalization of this study was recently proposed by C. Elhechmi \cite{chok1}. The author investigated an elliptic partial differential equation subject to a polynomial Robin boundary condition, established existence and uniqueness results for the associated solutions, and developed an iterative scheme for their construction.\\
  
In this work, we investigate problem $(S)$, characterized by an exponential-type boundary condition on  $\Gamma_R$. To establish existence and uniqueness results for solutions to problem $(S)$,  we employ an iterative scheme by constructing a sequence of linear problems $(S_k)_{k \in \mathbb{N}}$ designed to approximate $(S)$. We prove that each linearized problem $(S_k)$ admits a unique solution $u_k$, and that the sequence $(u_k)$ converges to the unique solution of the original nonlinear problem $(S)$.\\

The analysis relies on the application of suitable Sobolev embedding results in the periodic setting, together with the development of specific technical estimates. It should be noted that the uniqueness result is local, established within a specific closed ball of the Hilbert space $V$, the construction of which is detailed in the following section.\\

The paper is organized as follows. In Section~\ref{sec-prelm}, we present some preliminary results and state the main result of this work. In Section~\ref{sect-approx}, we define the sequence of linear problems $(S_k)_{k \in \mathbb{N}}$ and prove the existence and uniqueness of a solution to each problem $(S_k)$. In Section \ref{sect-cnvgce}, we provide a rigorous analysis of the convergence of the sequence of solutions $(u_k)$ to a limit $u$ within the Hilbert space $V$. In Sections~\ref{sect-exist} and~\ref{sec-uniq}, we establish the existence and uniqueness of a solution to problem $(S)$, respectively. Finally, in Section \ref{append}, we derive an explicit estimation of the Sobolev embedding constants relevant to our periodic setting.
\section{Preliminaries  and main result}\label{sec-prelm}
The linear operator is defined and continuous by
$$\tau_0:H^{1}(\Omega)\rightarrow L^2(\partial\Omega),\;u\mapsto u_{/\partial\Omega}.$$
Precisely, $\tau_0$ is continuous from $(H^{1}(\Omega),\|.\|_{H^{1}})$ to $(L^2(\partial\Omega),\|.\|_{L^2})$. The image of the operator $\tau_0$ is noted
$H^{1/2}(\partial\Omega)$. This space is endowed by the following norm 
$$\|\varphi\|_{H^{1/2}}=\inf\{\|u\|_{H^1}:\;u\in H^{1}(\Omega);\;u_{/\partial\Omega}=\varphi \}.$$
Then, we can define the following operator
$$\tilde{\tau}_0:(H^{1}(\Omega),\|.\|_{H^{1}})\rightarrow (H^{1/2}(\partial\Omega),\|.\|_{H^{1/2}}),\;u\mapsto u_{/\partial\Omega}.$$
Clearly, $\tilde{\tau}_0$ is continuous and $\|\tilde{\tau}_0\|\leq 1$. By using this definition of the space $H^{1/2}(\partial\Omega)$, we obtain the following embeddings. Precisely, we have the following classical result.
\begin{lemma}(\cite{Jer-Dron})\label{lem02} Let $\Omega\subset\R^2$ be a Lipschitzian domain. Then, for $p\in[2,+\infty)$ we have
$$H^{1/2}(\partial\Omega)\hookrightarrow L^p(\partial\Omega).$$
Precisely, there is a constant
$C_{\Omega,p} > 0$, depending only on $p$ and domain $\Omega$, such that
\begin{equation}\label{eq-in1}
\|u\|_{L^p(\partial\Omega)}\leq C_{\Omega,p} \|u\|_{H^{1/2}(\partial\Omega)},\;\forall u\in H^1(\Omega).
\end{equation}
and
\begin{equation}\label{eq-in11}\|f\|_{L^p(\partial\Omega)} \leq C_{\Omega,p} \|f\|_{H^1(\Omega)},\;\forall f\in H^1(\Omega).\end{equation}
\end{lemma}
In the present study, we extend the right-hand side of equation \eqref{eq-in1} to the fractional Sobolev space $H^s(\Gamma)$ for $s \in (0,1)$, ($\Gamma$ is nonempty subset of $S(0,1)$), and derive a refined estimate for the constant $C_{\Omega,p}$. To obtain precise inequalities with explicit constants, we provide  upper bounds for the embedding constant in the specific case where the domain is the unit disk, $\Omega = D(0,1)$.\\
Specifically, we establish the following result:
\begin{proposition}\label{th:ineg-sob1} 
Let $2 \leq p < \infty$. The embedding $H^{1/2}(S(0,1)) \hookrightarrow L^p(S(0,1))$ is continuous. In particular, for all $f \in H^{1/2}(S(0,1))$, we have:
\begin{equation}\label{in-12}
\|f\|_{L^p(S(0,1))} \leq \lambda_p \|f\|_{H^{1/2}(S(0,1))},
\end{equation}
where the embedding constant $\lambda_p$ is given by:
\begin{equation}\label{in-13}
\lambda_p =\left( C_1\frac{p}{p-2} C_2^{p-2} p^{\frac{p-2}{2}} \right)^{\frac{1}{p}}, \quad \forall p > 2,
\end{equation}
and $C_1,C_2$ denote some universal constants.
\end{proposition}
\begin{remark}\label{rem21}
\begin{enumerate}
\item If $p\geq3$, we can take
$$\lambda_p =\left(\mcc^{p-2} p^{\frac{p-2}{2}} \right)^{\frac{1}{p}},$$
where $\mcc=2\max(C_1,C_2)$.
\item A rigorous proof of Proposition \ref{th:ineg-sob1} is provided in {\rm{\bf Appendix C}}.
\item The proof of Proposition \ref{th:ineg-sob1} is done in two step:\begin{enumerate}
\item First step: We prove $H^s_{2\pi}\hookrightarrow L^p_{2\pi}$, with $\frac{1}{p}+s=\frac{1}{2}$.
\item Second step: We use $\|.\|_{H^s_{2\pi}}\leq \|.\|_{H^{1/2}_{2\pi}},\;\forall 0<s<1/2$.
\end{enumerate}
\end{enumerate}
\end{remark}
Now, we define the following Sobolev space $V$ adapted to our case. Let $V$ the space defined by
 $$V  =\{ v \in H^1(\Omega)\;  {\rm such\,  that }\;  v = 0\;{\rm on }\; \Gamma_D \}.$$
The space $V$ is a Hilbert space when endowed with the  inner product 
 $$\langle u, v\rangle_V = \ds\int_\Omega \nabla u .\nabla v \dx,$$  and the associated norm  $$\|u\|_V = \left(\ds\int_\Omega |\nabla u|^2dx\right)^\frac{1}{2}.$$
Let $\tau$ denote the trace operator defined by
$$
 \begin{array}{llll}
 \tau &: V &\rightarrow &H^{1/2}(\Gamma_R)\\
  &u &\mapsto & u_{|\Gamma_R}.
 \end{array}
 $$
The linear operator  $\tau$ is continuous when the space $H^{1/2}(\Gamma_R)$ is equipped by the norm $\|.\|_{L^p(\Gamma_R)}$, ($2\leq p<\infty$). Also, we define the constant
   $$\beta_p=\inf\big\{ C>0;\;\|u\|_{L^p(\Gamma_R)}\leq C \| u\|_{V},\,\,\, \forall u \in V\big\}\in(0,\infty).
   $$
\begin{remark} 
Proposition~\ref{th:ineg-sob1} plays a crucial role in the analysis of the nonlinear exponential boundary condition on $\Gamma_R$. Indeed, the embedding
$$
H^{1/2}\bigl(S(0,1)\bigr)\hookrightarrow L^p\bigl(S(0,1)\bigr),
\qquad 2 \le p < \infty,
$$
together with the continuity of the trace operator
$\tau : V \to H^{1/2}(\Gamma_R)$, implies that for every $u \in V$ and every
$p \ge 2$,
$$
\|u\|_{L^p(\Gamma_R)}
\le \beta_p \, \|u\|_V,
$$
where the constant $\beta_p$ can be explicitly controlled in terms of $\lambda_p$.\\
This estimate allows us to control polynomial growth terms on the boundary. More importantly, it provides a fundamental tool for handling the exponential nonlinearity appearing in the Robin boundary condition:
$$
\varphi(x)\left(e^{\alpha u} - e^{-(1-\alpha)u}\right).
$$
\end{remark}
{\bf Dependence of the constant $\beta_p$ on $\lambda_p$:} Recall that the constant $\beta_p$ is defined by
$$
\beta_p
= \inf \left\{ C>0 \;:\;
\|u\|_{L^p(\Gamma)} \le C \|u\|_V,\ \forall u\in V \right\}.
$$
The continuity of the operators
$$
\tau:V \longrightarrow H^{1/2}(\Gamma_R),\;u\longmapsto u_{/\Gamma_R},$$
and
$$H^{1/2}(\Gamma_R)\hookrightarrow L^p(\Gamma_R),
\qquad 2 \le p < \infty,
$$
implies that, for all $p>2$
\begin{equation}\label{eq-LpEstimate}
\|u\|_{L^p(\Gamma_R)}\le \lambda_p \|u\|_{H^{1/2}(\Gamma_R)}
\le \lambda_p \|\tau\| \|u\|_{V}.
\end{equation}
Therefore, the constant $\beta_p$ satisfies the estimate
\begin{equation}\label{eq-lambda--beta}
\beta_p \le \lambda_p \mt,
\end{equation}
where:
\begin{itemize}
\item $\mt = \|\tau\|$ denotes the norm of the trace operator
$\tau : V \to H^{1/2}(\Gamma_R)$.\\
By the definition $\|.\|_{H^{1/2}}$, we get $\|\mt\|\leq 1$, which implies $\beta_p \le \lambda_p$.
\item $\lambda_p$ is the constant given explicitly in Proposition~\ref{th:ineg-sob1}.
\end{itemize}
$$$$
Before stating the main result of this paper, it is necessary to provide some definitions and conditions:
\begin{enumerate}
\item[$\bullet$] $M_0=\beta_2(\|\phi\|_{L^2(\Gamma_N)}+\|g\|_{L^2(\Gamma_R)})$.\vspace{2mm}
\item[$\bullet$]  We denote by $\Phi_{ad} $ the set of admissible coefficient:
\begin{equation}\label{def-adm}
\Phi_{ad }=\left\{ h \in C^0( \overline{\Gamma_R}  );\;\; 0\leq h \leq{\xi} \Lambda \right\}.
\end{equation} 
 where   the parameter $\xi \in (0,1)$, and
\begin{equation}\label{def-Lambda}
\Lambda= \bigg[ \beta_3^3M_0 +\beta_4^4M_0^2+(CM_0)^3R(CM_0)\bigg]^{-1}.
\end{equation}
Note that the constant $C$, and the entire function $R$ will be defined in the proof of Lemma \ref{lem4.1}.
\vspace{2mm}
\end{enumerate}
We are now in a position to state the main result of this article.
\begin{theorem}
Problem $(S)$ admits a solution $u\in V$. Moreover, this solution  is unique in the  closed ball $\overline{B_V(0,M_0)} =\{ v \in V \,:\,\|v\|_V \leq M_0\}$.
\end{theorem}
\section{Study of approximation system} \label{sect-approx}
To approximate the solution of the nonlinear problem $(S)$, we define the following iterative scheme. For each  $k \geq 1$, let $(S_k)$ be the linear problem defined by:
$$
(S_{k})\left\{\begin{array}{lll}
\displaystyle\Delta u = 0, \;&{\rm in}&\;\Omega,\\
\displaystyle u = 0, \;&{\rm on}&\;\Gamma_D,\\
\displaystyle\frac{\partial u}{\partial n} = \phi(x)  \;&{\rm on }&\;\Gamma_N,\\
\displaystyle\frac{\partial u}{\partial n}+\varphi(x) f_\alpha(u_{k-1})u = g(x)  \;&{\rm on}&\;\Gamma_R.
\end{array}\right.
$$
with the initial iteration $u_0 = 0$.\\

Referring to \cite{{CJ1}}, we have the following result
\begin{lemma}\label{th:ex_0}
For $k=1$, the linear problem $(S_{1})$ admits a unique solution $u_1 \in V$. Furthermore, this solution is bounded in the $V$-norm by $M_0$:
\begin{equation}
\|u_1\|_V \leq M_0,
\end{equation}
where $M_0 = \beta_2\left( \|\phi\|_{L^2(\Gamma_N)} + \|g\|_{L^2(\Gamma_R)} \right)$.
\end{lemma}

In what follows, we state  the main result of this section. Precisely, in this proposition we establish the well-posed-ness of the iterative sequence and the uniform bounded-ness of its solutions.

\begin{proposition}\label{th:ex_k}
For each $k \in \mathbb{N}$, the linear problem $(S_{k})$ admits a unique solution $u_k \in V$. Furthermore, the sequence $(u_k)_{k \in \mathbb{N}}$ is uniformly bounded in $V$, satisfying:
\begin{equation}
\|u_k\|_V \leq M_0, \quad \forall k \in \mathbb{N},
\end{equation}
where $M_0$ is the constant defined in Lemma \ref{th:ex_0}.
\end{proposition}
{\bf Proof.} The proof relies on the variational formulation of  problem $(S_k)$ and the application of the  Lax-Milgram Theorem.\\ 
The problem $(S_{k})$ is formally equivalent to the following variational problem:
$$
(\tilde{S}_k) \left\{
\begin{array}{llll}
\mbox{Find } u \in V \mbox{ such that}\\\\
B_{\alpha,k}(u,v) =L(v) \,\, \forall  v \in V ,
\end{array}
\right.
$$
where
$$
B_{\alpha,k}(u,v) = \ds\int_{\Omega} \nabla u \cdot \nabla v \dx + \int_{\Gamma_R} \varphi(x) f_\alpha(u_{k-1}) u v \dsigma,
$$
and
$$
L(v) = \ds\int_{\Gamma_N} \phi v \dsigma + \int_{\Gamma_R} g v \dsigma.
$$

\noindent To establish the existence of a unique solution to  this variational problem, we apply the Lax-Milgram theorem. This requires verifying that the bilinear form $B_{\alpha,k}$ is both continuous and coercive on $V\times V$, and that the linear form $L$ is continuous on $V$.\\

\noindent\textbf{ Coercivity of $B_{\alpha,k}$:} For   $u \in V$, we have
	$$
	\begin{array}{llll}
	B_{\alpha,k}(u,u) &=& \ds\int_{\Omega} |\nabla u|^2 \dx + \int_{\Gamma_R} \varphi(x) f_\alpha(u_{k-1}) u^2 \dsigma \\
	&=&  \|  u\|_V^2 + \ds\int_{\Gamma_R} \varphi(x) f_\alpha(u_{k-1}) u^2 \dsigma.
	\end{array}
	$$
Since the positivity of the functions $\varphi$ and $f_\alpha$, we get 
	$$
	B_{\alpha,k}(u,u)  \geq \|u\|_{V}^2.
	$$
\noindent\textbf{Continuity of $B_{\alpha,k}$:} Let $(u,v)\in V^2$.\\ 
By applying the triangle inequality and the $L^\infty$ bound on $\varphi$, we have:
$$
\begin{array}{lllllllll}
|B_{\alpha,k}(u,v)| &\leq &\|\nabla u\|_{L^2(\Omega)} \|\nabla v\|_{L^2(\Omega)} + \|\varphi\|_{L^\infty({\Gamma_R})} \ds\int_{\Gamma_R} f_\alpha(u_{k-1})|u| |v| \dsigma.
\end{array}
$$
Using property {\bf(P3)}, we obtain
$$
\begin{array}{lllllllll}
|B_{\alpha,k}(u,v)| &\leq &\|  u\|_{V} \|  v\|_{V} + 2\delta\|\varphi\|_{L^\infty({\Gamma_R})} \ds\int_{\Gamma_R}  (1+ \delta|u_{k-1}|e^{\delta|u_{k-1}|})|u| |v| \dsigma\\
&\leq &\|u\|_{V} \|v\|_{V} + 2\delta\|\varphi\|_{L^\infty({\Gamma_R})}\bigg[\|u\|_{L^2(\Gamma_R)}\|v\|_{L^2(\Gamma_R)} + \ds\int_{\Gamma_R}  \delta|u_{k-1}||u| |v| e^{\delta|u_{k-1}|} \dsigma\bigg].\\
\end{array}
$$
Applying H\"older's inequality to the integral term yields:
$$
\begin{array}{lllllllll}
|B_{\alpha,k}(u,v)|&\leq &\|u\|_{V} \|v\|_{V} + 2\delta\|\varphi\|_{L^\infty({\Gamma_R})}\bigg[\|u\|_{L^2(\Gamma_R)}\|v\|_{L^2(\Gamma_R)} + \\
& &\|u\|_{L^3(\Gamma_R)} \|v\|_{L^3(\Gamma_R)} \bigg( \ds\int_{\Gamma_R}  (\delta|u_{k-1}|)^3  e^{3\delta|u_{k-1}|} \dsigma \bigg)^{1/3}\bigg].\\
\end{array}
$$
Let
$$
I_k = \ds\int_{\Gamma_R}   |u_{k-1}|^3  e^{3\delta|u_{k-1}|} \dsigma
$$
Using the power series expansion of the exponential function, we have
$$
\begin{array}{lllllllll}
I_k &= &\ds \int_{\Gamma_R}  |u_{k-1}|^{3} \sum_{n=0}^{\infty}\frac{(3\delta)^n}{n!}  |u_{k-1}|^{n} \dsigma\\
&=&\ds \sum_{n=0}^{\infty} \int_{\Gamma_R} \frac{(3\delta)^n}{n!}  |u_{k-1}|^{3+n} \dsigma\\
&=&\ds \sum_{n=0}^{\infty}  \frac{(3\delta)^n}{n!}  \|u_{k-1}\|_{L^{3+n}(\Gamma_R)}^{3+n}.
\end{array}
$$
Using the inequalities (\ref{eq-LpEstimate}, \ref{eq-lambda--beta}) and the uniform bound  $\|u_{k-1}\|_V \leq M_0$, we obtain

$$
\|u_{k-1}\|_{L^{3+n}(\Gamma_R)}^{3+n} \leq\lambda_{3+n}^{3+n}  {\mt}^{3+n} \|u_{k-1}\|_V^{3+n} \leq \lambda_{3+n}^{3+n} {\mt}^{3+n}M_0^{3+n}.
$$
By using the fact $\mt\leq1$, we get
$$
\begin{array}{llllll}
I_k &\leq& \ds\sum_{n=0}^{\infty} \frac{(3\delta)^n}{n!} \lambda_{3+n}^{3+n}M_0^{3+n}.
\end{array}$$
By using the explicit  definition of $\lambda_{3+n}$, we get
$$I_k \leq (CM_0)^3 \ds\sum^{\infty}_{n=0} (3\delta CM_0)^n a_n,$$
where $$a_n = \ds\frac{(n+3)^{\frac{n+1}{2}}}{n!}= \ds\frac{(n+3)^{\frac{n+1}{2}}}{n^{\frac{n+1}{2}}}\;\ds\frac{n^{\frac{n+1}{2}}}{n!}= \ds\left( 1 + \frac{3}{n} \right)^{\frac{n+1}{2}}\;\ds\frac{n^{\frac{n+1}{2}}}{n!}.$$ \\
To analyze the convergence of the series, we use the approximation $\ds\left( 1 + \frac{3}{n} \right)^{\frac{n+1}{2}} \thicksim e^{\frac{3}{2}}$, and the Stirling's formula:    $
	n! \sim n^{n+1/2}e^{-n}(2\pi)^{1/2}$ implies that
	$$
	a_n	\sim \ds\frac{e^n(2\pi)^{-1/2}}{n^{n/2}}.
	$$
	Let $ d_0=\max\Big(a_0,\ds\sup_{n\geq1}\frac{a_n}{\frac{e^n(2\pi)^{-1/2}}{n^{n/2}}}\Big)\in(0,\infty)$. Clearly, the series $\ds\sum_{n\geq1}\frac{e^n(2\pi)^{-1/2}}{n^{n/2}}(3\delta CM_0)^{n}$ converges, which implies that
	$$
	\begin{array}{llll}
		I_k	& \leq d_0 (CM_0)^3\Big[1+\ds\sum^{\infty}_{n=1} (3\delta C M_0)^n\frac{e^n(2\pi)^{-1/2}}{n^{n/2}}\Big]:=r_0<\infty.
	\end{array}
	$$
Finally
	$$\begin{array}{lcl}
	|B_{\alpha,k}(u,v)|&\leq&\left[ 1+2\delta\|\varphi\|_{L^\infty({\Gamma_R})}\left(\beta_2^2+\beta_3^2I_k^{1/3}\right)\right]\|u\|_V \|v\|_V  \vspace{2mm}\\
&\leq&\left[ 1+2\delta\varepsilon_0\left(\beta_2^2+\beta_3^2r_0^{1/3}\right)\right]\|u\|_V \|v\|_V,
\end{array}$$
This establishes the continuity of the bilinear form $B_{\alpha,k}$ on $V \times V$.\\

\noindent{\textbf{Continuity of $L$:}} The continuity of the linear functional $L$ follows from the Cauchy-Schwarz inequality and the continuity of the trace operator.\\
Precisely, for $u\in V$, we have
	$$
	\begin{array}{llll}
	|L(v)| &=& |\ds\int_{\Gamma_N} \phi v \dsigma + \int_{\Gamma_R} g v \dsigma| \vspace{2mm}\\
	&\leq& \|\phi\|_{L^2(\Gamma_N)}\|v\|_{L^2(\Gamma_N)} + \|g\|_{L^2(\Gamma_R)}\|v\|_{L^2(\Gamma_R)}\vspace{2mm}\\
	&\leq&\beta_2\left(\|\phi\|_{L^2(\Gamma_N)}+\| g\|_{L^2(\Gamma_R)}  \right)\|v\|_V.
	\end{array}
$$
 This confirms that $L$ is a bounded linear functional on $V$.\\

\noindent\textbf{Conclusion:} According to the Lax-Milgram Theorem, problem $(\tilde{S}_k)$ has a unique solution $u_k\in V$. Which implies that problem $(S_k)$ admits a unique solution  $u_k \in V$ for each $k\geq 1$.\\

 It remains to be establish the uniform bound-ness of the sequence $(u_k)_{k\in \N}$.\\
 For this, let $k\geq 1$. We have 
 $$ B_{\alpha,k}(u_k,u_k) = L(u_k),$$\\
 then
$$
\begin{array}{llll}
\ds\int_{\Omega} |\nabla u_k|^2 \dx + \int_{\Gamma_R} \varphi(x) f_\alpha(u_{k-1}) u_k^2 \dsigma = \ds\int_{\Gamma_N} \phi u_k \dsigma + \int_{\Gamma_R} g u_k \dsigma.
\end{array}
$$
Since the positivity of  functions $f_\alpha$ and $\varphi$, we get
$$
\|\nabla u_k\| ^2_{L^2(\Omega)} \leq \beta_2(\|\phi\|_{L^2(\Gamma_N)}+\|g\|_{L^2(\Gamma_R)})\|u_k\|_V
$$
then
$$
\|u_k\|_V  \leq M_0.
$$
\section{Convergence result}\label{sect-cnvgce}
The purpose of this section  is to prove that the sequence of   solutions $(u_k)_{k\geq1}$ of problems $(S_{k})_{k\geq1}$ converges to a solution of the original problem $(S)$. \\

For each $k \in \mathbb{N}$, let $u_k$ denote the unique solution of the problem
$$
(S_{k})\left\{\begin{array}{lllll}
\displaystyle\Delta u=  0, \;&{\rm in}&\;\Omega,\\
\displaystyle \,u=   0, \;&{\rm on}&\;\Gamma_D,\\
\displaystyle\frac{\partial u}{\partial n}=  \phi(x) \;&{\rm on }&\;\Gamma_N,\\
\displaystyle\frac{\partial u}{\partial n}+\varphi(x)  f_\alpha(u_{k-1})u =  g(x) \;&{\rm on}&\;\Gamma_R.
\end{array}\right.
$$
with initial iteration $u_0 =0$.\\
 
\noindent Let $w_k := u_{k+1} - u_k$, where $ u_{k+1}$ and $ u_k$ are respectively the solutions of  problems $(S_{k+1})$ and $(S_{k})$.\\ We obtain that $w_k$ solves  
 
$$
\left\{\begin{array}{lllll}
\displaystyle\Delta w_k&=&  0, \;&{\rm in}&\;\Omega,\\
\displaystyle \,w_k&=&   0, \;&{\rm on}&\;\Gamma_D,\\
\displaystyle\frac{\partial w_k}{\partial n}&=& 0 \;&{\rm on }&\;\Gamma_N,\\
\vspace{3mm}
\displaystyle\frac{\partial w_k}{\partial n}&+&\varphi(x)  \big[ f_{\alpha}(u_{k})u_{k+1} - f_{\alpha}(u_{k-1})u_{k}\big] =  0 \;&{\rm on}&\;\Gamma_R.
\end{array}\right.
$$
Using Green formula's, we obtain
$$
\ds \int_\Omega \nabla w_k  \nabla v - \ds \int_{\Gamma_R} \ds\frac{\partial w_k}{\partial n} v  = 0 \hspace{3mm} \forall v \in V.
$$
Choosing  $v =w_k$, it follows that
$$
 \|w_k \|_V^2  + \ds \int_{\Gamma_R} \varphi(x) \big[ f_{\alpha}(u_{k})u_{k+1} - f_{\alpha}(u_{k-1})u_{k}\big]w_k = 0.
$$
Replacing $u_{k+1} $ by $w_{k}+u_k $, we obtain
$$
 \|w_k \|_V^2  + \ds \int_{\Gamma_R} \varphi(x) \big[ f_{\alpha}(u_{k})w_k +  \big(f_{\alpha}(u_{k}) -f_{\alpha}(u_{k-1})\big)u_{k}\big]w_k = 0.
$$
Hence
$$
 \|w_k \|_V^2  + \ds \int_{\Gamma_R} \varphi(x) f_{\alpha}(u_{k})w_k^2+ \ds \int_{\Gamma_R} \varphi(x) \big[  f_{\alpha}(u_{k}) -f_{\alpha}(u_{k-1})\big] u_{k}w_k = 0.
$$
Since $\varphi(x) \geq 0$ for all $x \in \Gamma_{R}$, we deduce
\begin{equation}\label{eq1}
 \|w_k \|_V^2  \,\leq \,\ds \int_{\Gamma_R} \varphi(x) \big|f_{\alpha}(u_{k}) -f_{\alpha}(u_{k-1})\big|.|u_{k}|.|w_k|d\sigma.
\end{equation}
The Taylor expand of $f_\alpha(x)$ yields
$$
\begin{array}{lll}
f_\alpha(x) & = & \ds\frac{1}{x}\bigg( \sum_{m=0}^\infty \frac{(\alpha x)^m}{m!}- \sum_{n=0}^\infty (-1)^m\frac{(1-\alpha)^m x^m }{m!}\bigg)\vspace{2mm}\\
& = & \ds\frac{1}{x}\bigg( x+\sum_{m=2}^\infty b_m(\alpha) x^m\bigg)\vspace{2mm}\\
&=& 1 +\ds\sum_{m\geq 1} b_{m+1}(\alpha) x^m,
\end{array}
$$
where
$$\displaystyle b_{m}(\alpha) = \frac{\alpha^m - (-1)^m (1-\alpha)^m}{m!}.$$
 Consequently,
$$
f_\alpha(u_{k}) -f_\alpha(u_{k-1}) = b_{2}(\alpha)  w_{k-1} + b_{3}(\alpha)(u_k+u_{k-1})w_{k-1} + \ds\sum_{m\geq 3} b_{m+1}(\alpha)(u_k^m-u_{k-1}^m).$$
Using the elementary inequality
\begin{equation}\label{*}|x^m - y^m| \leq m |x-y| (|x|^{m-1} + |y|^{m-1})\,\, \forall x, y \in \R,\,\forall m \in \N,\end{equation}
we obtain
\begin{align}\label{eq2}
|f_\alpha(u_{k})-f_\alpha(u_{k-1})| &\leq  |b_{2}(\alpha)| | w_{k-1}| + |b_{3}(\alpha)| \left(|u_k | + |u_{k-1}|\right)| w_{k-1}| +\notag\\
& 
 \ds\sum_{m\geq 3} |b_{m+1}(\alpha)| m \left( |u_k |^{m-1} + |u_{k-1}|^{m-1}\right) | w_{k-1}|.
\end{align}
\begin{lemma}[Contraction estimate]\label{lem4.1}
There exists a constant $K \in (0,1)$ such that $(w_k)$ satisfies the following contraction estimate:
$$
\|w_{k}\|_V \leq K \|w_{k-1}\|_V,\;\forall k\in\N.
$$
\end{lemma}
\noindent{\bf Proof.} From inequalities \eqref{eq1} and \eqref{eq2}, we deduce
\begin{align}\label{eq:estim0}
\|w_k \|_V^2  \leq &\,\|\varphi\|_{L^\infty(\Gamma_R)}\ds \int_{\Gamma_R}   \big|f_\alpha(u_{k}) -f_\alpha(u_{k-1})\big|  |w_k | |u_{k}|d\sigma  \notag \\   %
\leq &\,\|\varphi\|_{L^\infty(\Gamma_R)}\big(J_{k,1}+J_{k,2}+J_{k,3}\big),                %
\end{align}
where
$$
\begin{array}{lcl}
J_{k,1} &= &\ds\int_{\Gamma_R}  |b_2(\alpha)|  |w_{k-1}| |w_k | |u_{k}|d\sigma \\
J_{k,2} &= &\ds\int_{\Gamma_R}  |b_3(\alpha)|  |w_{k-1}| \big(|u_k | + |u_{k-1}| \big) |w_{k }||u_{k}|d\sigma \\
J_{k,3} &= &\ds\sum_{m\geq3}|b_{m+1}(\alpha)| m \int_{\Gamma_R}\big(|u_k |^{m-1} + |u_{k-1}|^{m-1}\big) |w_{k-1}||w_k| |u_{k}|d\sigma.
\end{array}
$$
As $\alpha \in (0,1)$, then
$$ \displaystyle |b_{m}(\alpha)| = |\frac{\alpha^m - (-1)^m (1-\alpha)^m}{m!}|\,\leq\, \ds\frac{2}{m!},\;\forall m \in \N.$$
\noindent $\bullet$ \textbf{Estimation of $J_{k,1}$:} Applying H\"older's inequality on $\Gamma_R$ with exponents $(3,3,3)$, we obtain:
\begin{align}\label{eq:estimJ1}
\begin{array}{llll}
J_{k,1} &\leq&  |b_2(\alpha)| \|w_{k-1}\|_{L^3(\Gamma_R)} \|w_k \|_{L^3(\Gamma_R)} \|u_{k}\|_{L^3(\Gamma_R)}\vspace{2mm}\\
&\leq& |b_2(\alpha)|\beta_3^3M_0\|w_{k-1}\|_V\|w_{k}\|_V.\vspace{2mm}\\
&\leq& \beta_3^3M_0\|w_{k-1}\|_V\|w_{k}\|_V,
\end{array}
\end{align}
where we have used  the uniform bound $\|u_k\|_V\le M_0$ and the estimate $ |b_2(\alpha)|\le 1$.

\noindent $\bullet$ \textbf{Estimation of $J_{k,2}$:} Applying H\"older's inequality on $\Gamma_R$ with exponents $(4,4,4,4)$ together with the trace embedding, we obtain
\begin{align}\label{eq:estimJ2}
\begin{array}{llll}
J_{k,2} &\leq&  |b_3(\alpha)| \|w_{k-1}\|_{L^4(\Gamma_R)} \|w_k \|_{L^4(\Gamma_R)} (\|u_{k}\|_{L^4(\Gamma_R)} +\|u_{k-1}\|_{L^4(\Gamma_R)})\|u_{k}\|_{L^4(\Gamma_R)} \vspace{2mm}\\
&\leq& 2|b_3(\alpha)|\beta_4^4M_0^2\|w_{k-1}\|_V\|w_{k}\|_V\vspace{2mm}\\
&\leq& \beta_4^4M_0^2\|w_{k-1}\|_V\|w_{k}\|_V.
\end{array}
\end{align}
$\bullet$ \textbf{Estimation of $J_{k,3}$:} Applying H\"older's inequality on $\Gamma_R$ with exponents $(m+2,m+2,m+2)$
We have
$$
\begin{array}{llll}
J_{k,3} &\leq& \ds\sum_{m\geq 3} |b_{m+1}(\alpha)| m\bigg[ \ds \int_{\Gamma_R}     |u_k |^{m}  | w_{k}| | w_{k-1}| +  \ds \int_{\Gamma_R}    | u_{k-1}|^{m-1}  |u_k |  | w_{k}| | w_{k-1}| \bigg]
\vspace{2mm}\\
&\leq& \ds\sum_{m\geq 3} |b_{m+1}(\alpha)| m\bigg[  \|u_{k}\|_{L^{m+2}(\Gamma_R)}^m\|w_{k-1}\|_{L^{m+2}(\Gamma_R)} \|w_k \|_{L^{m+2}(\Gamma_R)}   \\
& &+   \|u_{k-1}\|_{L^{m+2}(\Gamma_R)}^{m-1}\|u_{k}\|_{L^{m+2}(\Gamma_R)}\|w_{k-1}\|_{L^{m+2}(\Gamma_R)} \|w_k \|_{L^{m+2}(\Gamma_R)} \bigg].
\end{array}
$$
Using the trace embedding and the uniform bound $\|u_k\|_V \leq M_0$, we get
$$
\begin{array}{llll}
J_{k,3} 
&\leq&  \ds\sum_{m\geq 3} |b_{m+1}(\alpha)| m\bigg[  \beta_{m+2}^{m+2} M_0^m +  \beta_{m+2}^{m+2}M_0^m\bigg]\|w_{k-1}\|_V\|w_{k}\|_V
\vspace{2mm}\\
&\leq&  \ds\sum_{m\geq 3} 2 |b_{m+1}(\alpha)| m   \beta_{m+2}^{m+2} M_0^m   \|w_{k-1}\|_V\|w_{k}\|_V.
\end{array}
$$
Since $|b_{m}(\alpha)|   \leq \ds\frac{2}{m!}$, $\mt\leq1$, $\beta_{m+2}\leq \lambda_{m+2}$ and using the inequalities (\ref{eq-LpEstimate})-(\ref{eq-lambda--beta}), we get 
\begin{align}\label{eq:estimJ3}
\begin{array}{llll}
J_{k,3} &\leq& \bigg[ \ds\sum_{m\geq 3} \ds\frac{4}{(m+1)!}  m   \lambda_{m+2}^{m+2} M_0^m \mt^m  \bigg] \|w_{k-1}\|_V\|w_{k}\|_V\vspace{2mm}\\
&\leq&  \bigg[\ds\sum_{m\geq 3} \ds\frac{4}{(m+1)!} (m+2)^{\frac{m+2}{2}}\left( C M_0\right)^m   \bigg] \|w_{k-1}\|_V\|w_{k}\|_V\vspace{2mm}\\
&\leq&  \left( C M_0\right)^3\bigg[\ds\sum_{m\geq 3} q_m\left(CM_0\right)^{m-3}\bigg] \|w_{k-1}\|_V\|w_{k}\|_V,
\end{array}
\end{align}
where $q_m = \ds\frac{4}{(m+1)!} (m+2)^{\frac{m+2}{2}}$. Now we define the entire function(See {\rm{\bf Appendix B}}) $$R(z)=\sum_{m=3}^\infty q_mz^{m-3},\;z\in\C.$$
Hence,
\begin{equation}\label{eq:estimJ3-fin}
J_{k,3}  \leq(C M_0)^3R(CM_0)\|w_{k-1}\|_V\|w_{k}\|_V.
\end{equation}
Using  inequalities \eqref{eq:estim0}, \eqref{eq:estimJ1}, \eqref{eq:estimJ2} and \eqref{eq:estimJ3-fin}, we obtain 
$$
\|w_k \|_V^2  \leq \|\varphi\|_{L^\infty(\Gamma_R)}\bigg[\beta_3^3M_0+\beta_4^4M_0^2+(CM_0)^3R(CM_0)\bigg] \|w_{k-1}\|_V\|w_{k}\|_V
$$
and therefore
$$
\|w_k \|_V   \leq \|\varphi\|_{L^\infty(\Gamma_R)}\bigg[ \beta_3^3M_0 + \beta_4^4M_0^2  +  (CM_0)^3R(CM_0)\bigg] \|w_{k-1}\|_V.
$$
By the definition of the admissible set $\Phi_{ad}$ in \eqref{def-adm}, we have 
$$
\|\varphi\|_{L^\infty(\Gamma_R)} < \xi \Lambda \quad \text{ with } \xi\in (0,1),
$$
where  $\Lambda$  is  defined in \eqref{def-Lambda} by:
$$
\Lambda= \bigg[ \beta_3^3M_0 + \beta_4^4M_0^2+(CM_0)^3R(CM_0)\bigg]^{-1}.
$$
Consequently,
\begin{align}
\|w_k \|_V &\leq \xi\Lambda \bigg[ \beta_3^3M_0 + \beta_4^4M_0^2+(CM_0)R(CM_0)\bigg] \|w_{k-1}\|_V\notag \\  
  & \leq K \|w_{k-1}\|_V,     
\end{align}
where $K = \xi \in (0,1)$. Therefore, the sequence $(w_k)_{k\in\N}$ is a contraction in the space $V$. In particular, there exists a constant $K\in(0,1)$ such that 
\begin{equation}
\|w_{k}\|_V \leq K \|w_{k-1}\|_V \qquad k \in \N.
\end{equation}
 This complete the proof of Lemma \ref{lem4.1}.
\begin{remark}
We now complete the proof of the convergence result. Since the sequence $(w_k)_{k\geq1}$ satisfies a contraction estimate in the Banach space $V$, the series $\sum_{k\geq1} w_k$ converges absolutely in $V$. Noting  that $u_k = u_1 + \sum_{j=1}^{k-1}w_j$, it follows that the sequence $(u_k)_{k\geq1}$ is Cauchy in $V$ and converges strongly in $V$ to some limit $u \in V$.
\end{remark}
\section{Existence of a solution of the problem $(S)$}\label{sect-exist}
In this section, we prove that the limit
$$u=\lim_{k\rightarrow\infty}u_k\in V$$ 
is a solution to problem $(S)$.\\
Since $u_k\to u$ strongly in $V$, and each $u_k$ solves the linear problem $(S_k)$ we can pass to limit in the interior of domain $\Omega$ and on tha parts of the  boundary $\Gamma_D$ and $\Gamma_N$. \\
More precisely, we have
$$\begin{array}{lcl}
\Big(\forall k\geq1:\;\;\Delta u_k=0\;\;{\rm in}\;\;\Omega\Big)&\Longrightarrow&\Delta u=0\;\;{\rm in}\;\;\Omega,\\
\Big(\forall k\geq1:\;\;u_k=0\;\;{\rm on}\;\;\Gamma_D\Big)&\Longrightarrow&u=0\;\;{\rm on}\;\;\Gamma_D,\\
\Big(\forall k\geq1:\;\;\ds\frac{\partial u_k}{\partial n}=0\;\;{\rm on}\;\;\Gamma_N\Big)&\Longrightarrow&\ds\frac{\partial u}{\partial n}=0\;\;{\rm on}\;\;\Gamma_N.
\end{array}$$
It remains to identify the limit in the nonlinear Robin boundary condition. To this end, it suffices to prove that
$$\lim_{k\rightarrow\infty}\| f_\alpha(u_{k-1})u_k- f_\alpha(u)u\|_{L^1(\Gamma_R)}=0.$$
For $k\geq1$ set $\tilde{w}_k=u_k-u$. we have
$$\begin{array}{lcl}
\| f_\alpha(u_{k-1})u_k- f_\alpha(u)u\|_{L^1(\Gamma_R)}&\leq&\|f_\alpha(u_{k-1})\tilde{w}_k+(f_\alpha(u_{k-1})-f_\alpha(u))u\|_{L^1(\Gamma_R)}\\
&\leq&\underbrace{\|f_\alpha(u_{k-1})\tilde{w}_k\|_{L^1(\Gamma_R)}}+\underbrace{
\|(f_\alpha(u_{k-1})-f_\alpha(u))u\|_{L^1(\Gamma_R)}}\\
&&\quad\quad\quad\quad\quad\;\shortparallel\quad\quad\quad\quad\quad\quad\quad\quad\quad\quad\quad\quad\quad\shortparallel\\
&&\quad\quad\quad\quad\quad\;X_k\quad\quad\quad\quad\quad\quad\quad\quad\quad\quad\quad\quad\;Y_k\\
\end{array}$$
\begin{enumerate}
\item[$\bullet$] {\bf Estimate of $X_k$:} Using the series expansion of $f_\alpha$ and H\"older's inequality, we obtain 
$$\begin{array}{lcl}
X_k&\leq&\displaystyle \int_{\Gamma_R}|\tilde{w}_k|+\sum_{m=2}^{\infty}|b_{m+1}(\alpha)|\int_{\Gamma_R}|u_{k-1}|^m|\tilde{w}_k|\\
&\leq&\displaystyle (\sigma(\Gamma_R))^{1/2}\|\tilde{w}_k\|_{L^2(\Gamma_R)}
+\sum_{m=2}^{\infty}|b_{m+1}(\alpha)\|u_{k-1}\|_{L^{m+1}(\Gamma_R)}^m\|\tilde{w}_k\|_{L^{m+1}(\Gamma_R)}\\
&\leq&\displaystyle (\sigma(\Gamma_R))^{1/2}\beta_2\|\tilde{w}_k\|_{V}
+\sum_{m=2}^{\infty}|b_{m+1}(\alpha)\beta_{m+1}^{m+1}\|u_{k-1}\|_{V}^m\|\tilde{w}_k\|_{V}\\
&\leq&\displaystyle \Big((\sigma(\Gamma_R))^{1/2}\beta_2
+\sum_{m=2}^{\infty}|b_{m+1}(\alpha)\lambda_{m+1}^{m+1}
M_0^m\Big)\|\tilde{w}_k\|_{V},\;(\beta_{m+1}\leq \lambda_{m+1}).
\end{array}$$
Using the same technique employed in the previous proofs, we can deduce that the series $\displaystyle \sum_{m\geq2}|b_{m+1}(\alpha)\beta_{m+1}^{m+1}M_0^m$ is convergent and $\|\tilde{w}_k\|_{V}\rightarrow0$, then 
$$\displaystyle\lim_{k\rightarrow\infty}X_k=0.$$
\item[$\bullet$]  {\bf Estimate of $Y_k$:} By combining the elementary inequality (\ref{*}) and H\"older inequality, we get 
$$\begin{array}{lcl}
Y_k&\leq&\displaystyle \sum_{m=2}^{\infty}|b_{m+1}(\alpha)|\int_{\Gamma_R}|u_{k-1}^m-u^m|.|u|\\
&\leq&\displaystyle \sum_{m=2}^{\infty}|b_{m+1}(\alpha)|\int_{\Gamma_R}m(|u_{k-1}|^{m-1}+|u|^{m-1}).|\tilde{w}_{k-1}|.|u|\\
&\leq&\displaystyle \sum_{m=2}^{\infty}m|b_{m+1}(\alpha)|
\Big(\int_{\Gamma_R}|u_{k-1}|^{m-1}|\tilde{w}_{k-1}|.|u|+\int_{\Gamma_R}|u|^{m}|.|\tilde{w}_{k-1}|\Big)\\
&\leq&\displaystyle \sum_{m=2}^{\infty}m|b_{m+1}(\alpha)|
\Big(\|u_{k-1}\|_{L^{m+1}(\Gamma_R)}^{m-1}\|\tilde{w}_{k-1}\|_{L^{m+1}(\Gamma_R)}\|u\|_{L^{m+1}(\Gamma_R)}
+\|u\|_{L^{m+1}(\Gamma_R)}^{m}|\|\tilde{w}_{k-1}\|_{L^{m+1}(\Gamma_R)}\Big).
\end{array}$$
By using the Sobolev embedding and the bound-ness of $\|u_{k-1}\|_{V}$ and $\|u\|_{V}$, we obtain
$$\begin{array}{lcl}
Y_k&\leq&\displaystyle \sum_{m=2}^{\infty}m|b_{m+1}(\alpha)|^{m+1}\lambda_{m+1}^{m+1}\Big(\|u_{k-1}\|_{V}^{m-1}\|u\|_{V}
+\|u\|_{V)}^{m}\Big)\|\tilde{w}_{k-1}\|_{V}\\
&\leq&\displaystyle \Big(\sum_{m=2}^{\infty}2m|b_{m+1}(\alpha)|^{m+1}\lambda_{m+1}^{m+1}M_0^{m}\Big)\|\tilde{w}_{k-1}\|_{V}.\\
\end{array}$$
Using the same technique employed in the previous proofs, we can deduce that the series $
\ds\sum_{m\geq2} 2 m |b_{m+1}(\alpha)|\lambda_{m+1}^{m+1} M_0^m$ 
converges and $\|\tilde w_{k-1}\|_V \to 0$ as $k \to \infty$. It follows that
$$
\lim_{k\to\infty} Y_k = 0.
$$
Combining the estimates for $X_k$ and $Y_k$, we conclude that
$$
\lim_{k\to\infty} \| f_\alpha(u_{k-1}) u_k - f_\alpha(u) u \|_{L^1(\Gamma_R)} = 0.
$$
Hence, passing to the limit in the nonlinear Robin boundary condition gives
$$
\frac{\partial u}{\partial n} + \varphi(x) f_\alpha(u) u = g(x) \quad \text{on } \Gamma_R,
$$
which shows that $u$ is indeed a solution of the problem $(S)$.
\end{enumerate}
\section{Uniqueness of solution of the problem $(S)$ in $\label{sec-uniq}\overline{B_V(0,M_0)}$}
Let $v_1,v_2\in V$ such that $\|v_1\|_V,\|v_2\|_V\leq M_0$  be two solution of the problem $(S)$.\\
Set $w=v_1-v_2$. By Green's formula, we have
$$\|w\|_V^2\leq \int_{\Gamma_R}\varphi(x)|f_\alpha(v_1)v_1-f_\alpha(v_2)v_2|.|w|\leq\Lambda \int_{\Gamma_R}|f_\alpha(v_1)v_1-f_\alpha(v_2)v_2|.|w|$$
 using the decomposition
$$f_\alpha(v_1)v_2-f_\alpha(v_2)v_2=f_\alpha(v_1)w+(f_\alpha(v_1)-f_\alpha(v_2))v_2$$
we obtain
$$\|w\|_V\leq \Lambda(I+J),$$
where
$$\begin{array}{lcl}
I&=&\displaystyle\int_{\Gamma_R}|f_\alpha(v_1)|.|w|^2\vspace{2mm}\\
J&=&\displaystyle\int_{\Gamma_R}|f_\alpha(v_1)-f_\alpha(v_2)|.|v_2|.|w|.
\end{array}$$
By applaying the same technique used for the convergence of of the sequence $(u_k)$ to $u$, we can estimate $I$ and $J$ to get
$$\|w\|_V^2\leq \tilde{K}\|w\|_V^2$$
with $\tilde{K}<1$.\\
Hence, it follows that $w=0$ which implies $v_1=v_2$.
\section{Appendices}\label{append}
\subsection{Appendix A} \label{append-A} In this subsection, we prove the property {\bf(P3)} of the function $f_\alpha$:\\
For any $r \in \R^*$, we can write $f_\alpha(r)$ as a series:
 $$
 \begin{array}{llllllll}
0&\leq&f_\alpha(r)&=& \ds\sum_{n=1}^{\infty}\frac{\alpha^n-(-1)^n(1-\alpha)^n}{n!}r^{n-1}\\
& &  &\leq&\ds\sum_{n=1}^{\infty}\frac{|\alpha^n-(-1)^n(1-\alpha)^n|}{n!}|r|^{n-1}\\
& &  &\leq&\ds\sum_{n=1}^{\infty}\frac{2\delta^n}{n!} |r|)^{n-1}\\
& &  &\leq&2\delta\ds\sum_{n=1}^{\infty}\frac{(\delta |r|)^{n-1}}{n!} \\
& &  &\leq&2\delta\bigg(1+ \ds\sum_{n=2}^{\infty}\frac{(\delta |r|)^{n-1}}{n!} \bigg)\\
& &  &\leq&2\delta\bigg(1+ \delta|r|\ds\sum_{n=2}^{\infty}\frac{(\delta |r|)^{n-2}}{n!} \bigg)\\
& &  &\leq&2\delta\bigg(1+ \delta|r|\ds\sum_{n=2}^{\infty}\frac{(\delta |r|)^{n-2}}{(n-2)!} \bigg)\\
& &  &\leq&2\delta(1+ \delta|r|e^{\delta|r|}),
 \end{array}
 $$
which complete this proof.
\subsection{Appendix B}\label{append-B} In this part of Appendix, we prove that the function 
$$R(z)=\sum_{m=3}^\infty q_mz^{m-3},\;q_m=\frac{4}{(m+1)!}(m+2)^{\frac{m+2}{2}},$$
is an entire function.\\
For this, we use Stirling formula, we get
$$\begin{array}{lcl}
q_m&=&\ds4(m+2).\frac{1}{(m+2)!}(m+2)^{\frac{m+2}{2}}\\
&\sim&\ds4(m+2).\frac{1}{(m+2)^{m+2+\frac{1}{2}}e^{-m-2}(2\pi)^{1/2}}(m+2)^{\frac{m+2}{2}}\\
&\sim&\ds\frac{4e^{m+2}(2\pi)^{-1/2}}{(m+2)^{\frac{m+1}{2}}}.
\end{array}$$
Clearly, by the root test, we can conclude that the series is convergent for all $z$ in $\C$. Which complete the proof.
\subsection{Appendix C: Sobolev constants in Fourier series}\label{append-C}
In this appendix, we recall some necessary definitions and results to establish an improvement of Lemma \ref{lem02}.\\
We begin by proving a technical lemma which we use to demonstrate Sobolev's inequality in the periodic case.
\begin{lemma}\label{lem1} For $\theta\in(0,1)$, we have
$$\sum_{k=1}^Nk^{-\theta}\leq 1+\frac{N^{1-\theta}}{1-\theta}\leq 2\frac{N^{1-\theta}}{1-\theta},\;\forall N\in\N.$$
\end{lemma}
{\bf Proof.} It suffices to compare $\ds\sum_{i=1}^N(i+1)^{-\theta}$ to the integral $\ds\int_1^Nt^{-\theta}dt$.\\
\noindent Now we introduce the periodic Lebesgue spaces and the periodic Sobolev spaces, which will be used to estimate the Sobolev constants in the context of Fourier series.
\begin{definition}(\cite{QZ1})
\begin{enumerate}
\item  Let $f,g:\R\rightarrow\C$ be two measurable and $2\pi$-periodic functions.
We denote
$$f\sim g\Longleftrightarrow f=g\;a.e.$$
The relation $``\sim''$ is well defined and it is equivalence relation.\\
\item  For $p\in[1,\infty)$, we denote
$$L^p_{2\pi}=\{f:\R\rightarrow\C\;measurable\;and\;2\pi-periodic\;:\int_{[0,2\pi]}|f(t)|^pdt<\infty\}/\sim.$$
For $f\in L^p_{2\pi}$, we denote
$$\|f\|_{L^p_{2\pi}}=\Big(\frac{1}{2\pi}\int_{[0,2\pi]}|f(t)|^pdt\Big)^{1/p}.$$
\item  The space $L^\infty_{2\pi}$ is defined by
$$L^\infty_{2\pi}=\{f:\R\rightarrow\C\;measurable\;and\;2\pi-periodic\;:\exists\,C\geq0;\;|f(t)|\leq C,\;a.e\}/\sim.$$
For $f\in L^\infty_{2\pi}$, we denote
$$\|f\|_{L^\infty_{2\pi}}=\inf\{C\geq0;\;|f(t)|\leq C,\;a.e\}.$$
\item $C_{2\pi}=\{f:\R\rightarrow\C\; continuous\;and\;2\pi- periodic\}$.
\item $C_{2\pi}^\infty=\{f:\R\rightarrow\C\; C^\infty\;and\;2\pi- periodic\}$.
\item $\ds\mathcal P_{2\pi}=\{f=\sum_{-N\leq k\leq N}c_ke^{ikx},\;c_k\in\C\}$: The space of trigonometric polynomial functions.
\item If $f\in L^1_{2\pi}$, the Fourier coefficients of $f$ are defined by
$$\widehat{f}(k)=\frac{1}{2\pi}\int_0^{2\pi}f(t)e^{-ikt}dt,\;k\in\Z.$$
\item $\displaystyle C_0(\Z)=\{(z_k)_{k\in\Z}\in\C^\Z:\;\lim_{k\rightarrow\pm\infty}z_k=0\}.$
\end{enumerate}
\end{definition}

\begin{remark}(\cite{QZ1})
\begin{enumerate}
\item Let $1\leq p\leq \infty$. The normed space $(L^p_{2\pi},\|.\|_{L^p_{2\pi}})$ is complete.
\item Let $1\leq p<\infty$. $C^\infty_{2\pi}$ is dense in the space $(L^p_{2\pi},\|.\|_{L^p_{2\pi}})$.
\item Let $1\leq p\leq q\leq\infty$. we have $L^{q}_{2\pi}\hookrightarrow L^{p}_{2\pi}$. Precisely, we have
$$\|f\|_{L^{p}_{2\pi}}\leq \|f\|_{L^{q}_{2\pi}},\;\forall f\in L^{q}_{2\pi}.$$
\item $\mathcal P_{2\pi}\subset C^\infty_{2\pi}\subset C_{2\pi}$.
\item $\mathcal P_{2\pi}$ is dense in $L^p_{2\pi}$ for all $p\in[1,\infty)$.
\item The operator $\displaystyle\mathcal F:L^1_{2\pi}\rightarrow C_0(\Z),\;f\mapsto(\widehat{f}(k))_{k\in\Z}$ is well defined and injective.
\item The operator $\displaystyle\mathcal F:L^2_{2\pi}\rightarrow l^2(\Z),\;f\mapsto(\widehat{f}(k))_{k\in\Z}$ is well defined and unitary.
\end{enumerate}
\end{remark}
In the following we define the periodic Sobolev space and we give some properties, for more details and proofs see \cite{HE1, AU1} and \cite{BO1}.
\begin{definition} For $s\geq0$, the periodic Sobolev space of index $s$ is defined by
$$H^s_{2\pi}=\{f\in L^2_{2\pi}:\;\sum_{k\in\Z}(1+|k|^2)^s|\widehat{f}(k)|^2<\infty\}.$$
$H^s_{2\pi}$ is equipped by the inner product
$$\langle f/g\rangle_{H^s_{2\pi}}=\sum_{k\in\Z}(1+|k|^2)^s\widehat{f}(k)\overline{\widehat{g}(k)}$$
and the associated norm
$$\|f\|_{H^s_{2\pi}}=\Big(\sum_{k\in\Z}(1+|k|^2)^s|\widehat{f}(k)|^2\Big)^{1/2}.$$
\end{definition}

\begin{proposition} We have
\begin{enumerate}
\item For $s\geq0$: $(H^s_{2\pi},\langle ./.\rangle_{H^s_{2\pi}})$ is an Hilbert space.
\item For $s_2\geq s_1\geq0$, we have $H^{s_2}_{2\pi}\hookrightarrow H^{s_1}_{2\pi}$. Precisely, we have
$$\|f\|_{H^{s_1}_{2\pi}}\leq \|f\|_{H^{s_2}_{2\pi}},\;\forall f\in H^{s_2}_{2\pi}.$$
\end{enumerate}
\end{proposition}
Now, we give the main result statement for this section.
\begin{proposition}\label{th:ineg-sob72} Let $2<p<\infty$ and $s\in(0,1/2)$ such that  $  \ds\frac{1}{p}+\frac{s}{1}=\frac{1}{2} $,
then $H^s_{2\pi}\hookrightarrow L^p_{2\pi}.$
Precisely, we have
  \begin{equation}
  \|f\|_{L^p_{2\pi}} \leq 2{\tilde R}_p \|f\|_{H^s_{2\pi}},
  \end{equation}
with
\begin{equation}
{\tilde R}_p^p = C_1\frac{p}{p-2} C_2^{p-2} p^{\frac{p-2}{2}},
\end{equation}
  and $C_1,C_2 $ are universal constants.
\end{proposition}
\begin{remark} Combining this result with Remark \ref{rem21} we get: for $p\in(2,\infty)$
 $$
  \|f\|_{L^p_{2\pi}} \leq \lambda_p \|f\|_{H^{1/2}_{2\pi}},\;\forall f\in H^{1/2}_{2\pi},
$$
with $\lambda_p=2{\tilde R}_p$.
\end{remark}
{\bf Proof of Proposition \ref{th:ineg-sob72}.} This proof is inspired by the classical case proof in (\cite{BCD1}). Our method of study requires us to distinguish three cases:
\begin{enumerate}
\item[$\bullet$] {\bf First case:} Let  $f\in C^\infty_{2\pi}$ such that $\widehat{f}(0)=0$.\\
We have
$$2\pi\|f\|_{L^p_{2\pi}}^p=p\int_0^\infty t^{p-1}\lambda_1(\{x\in[0,2\pi]:\;|f(x)|>t\})dt.$$
We will decompose $f$ in such a way as to reveal norms of type $L^2$. To do this, let us split $f$ into "low" and "high" frequencies by setting, for $N\in\N_0$,
$$f=f_{1,N}+f_{2,N}$$
with
$$\begin{array}{lll}f_{1,N}&=&\mathcal F^{-1}\big({\bf 1}_{\{|k|< N\}}\widehat{f}(k)\big)\\\\
f_{2,N}&=&\mathcal F^{-1}\big({\bf 1}_{\{|k|\geq N\}}\widehat{f}(k)\big).\\
\end{array}$$
Since the support of the Fourier transform of $f_{1,N}$ is finite, the function $f_{1,N}$ is bounded and more precisely, we have\\
$-$ If $N=0$, then $f_{1,N}=0$\\
$-$ If $N\geq1$, we can write
$$\begin{array}{lll}
2\pi\|f_{1,N}\|_{L^\infty_{2\pi}}&\leq& \displaystyle\sum_{k=-(N-1)}^{N-1}|\widehat{f}(k)|\\
&\leq&\displaystyle\sum_{k=-(N-1)}^{N-1}(1+k^2)^{-s/2}(1+k^2)^{s/2}|\widehat{f}(k)|\\
&\leq&\displaystyle\Big(\sum_{k=-(N-1)}^{N-1}(1+k^2)^{-s}\Big)^{1/2}
\Big(\sum_{k=-(N-1)}^{N-1}(1+k^2)^{s}|\widehat{f}(k)|^2\Big)^{1/2}\\
&\leq&\displaystyle\Big(\sum_{k=-(N-1)}^{N-1}(1+k^2)^{-s}\Big)^{1/2}
\|f\|_{H^s_{2\pi}}.
\end{array}$$
Let $S_N=\displaystyle\sum_{k=-(N-1)}^{N-1}(1+k^2)^{-s} \leq \displaystyle2\sum_{k=0}^{N-1}(1+k^2)^{-s}.
$\\
For $k\geq0$, we have  
$$1+k^2=\displaystyle1+ \frac{1}{2}k^2+\frac{1}{2}k^2\geq\displaystyle \frac{1}{4}+\frac{1}{4}k^2+\frac{1}{2}k\geq\displaystyle \frac{1}{4}(1+k)^2,$$
so
$$(1+k^2)^{-s}\leq 4^s(1+k)^{-2s}.$$
Hence, for $N\geq1$, we get $$S_N\leq\displaystyle4^s\sum_{k=0}^{N-1}(1+k)^{-2s}
\leq\displaystyle4^s\sum_{k=1}^{N}k^{-2s}
\leq\displaystyle4\sum_{k=1}^{N}k^{-2s}.$$
By Lemma \ref{lem1}, we get
$$S_N\leq 8\frac{N^{1-2s}}{1-2s}.$$
Then, in both cases $N=0$ and $N\geq1$ we get
\begin{equation}\label{eq1:ineg-sob}
2\pi\|f_{1,N}\|_{L^\infty_{2\pi}}\leq 2\sqrt{2}(1-2s)^{-1/2}N^{\frac{1}{2}-s}\|f\|_{H^s_{2\pi}}.
\end{equation}
\noindent By applying the triangular inequality, we have, for any integer $N\geq0$,
$$\{|f|>t\}\subset \{|f_{1,N}|>t/2\}\cup \{|f_{2,N}|>t/2\}.$$

\noindent Using the inequality (\ref{eq1:ineg-sob}) above, if we choose $N=N_t\in\N_0$ with

$$N_t\leq \Big(\frac{t(1-2s)^{1/2}}{4C\|f\|_{H^s_{2\pi}}}\Big)^p<N_t+1$$
we obtain
$$\|f_{1,N_t}\|_{L^\infty_{2\pi}}\leq\frac{t}{4}
\Longrightarrow \lambda_1\big(\{|f_{1,N_t}|>t/2\}\big)=0.$$
Therefore, we can deduce that
\begin{equation}\label{eq-ineq11}
2\pi\|f\|_{L^p_{2\pi}}^p\leq p\int_0^\infty t^{p-1}\lambda_1(\{2|f_{2,N_t}|>t\})dt.
\end{equation}
It is well known that
$$\begin{array}{lll}
\lambda_1\Big(\{|f_{2,N_t}|>t/2\}\Big)&=&\displaystyle\int_{\{|f_{2,N_t}|>t/2\}}1\;dx\\
&\leq&\displaystyle\int_{\{|f_{2,N_t}|>t/2\}}\frac{4|f_{2,N_t}(x)|^2}{t^2}dx\\
&\leq&\displaystyle4\frac{\|f_{2,N_t}\|_{L^2_{2\pi}}^2}{t^2}.
\end{array}$$
By the equation (\ref{eq-ineq11}), we obtain
\begin{equation}\label{eq2:ineg-sob}
2\pi\|f\|_{L^p_{2\pi}}^p\leq 4p\int_0^\infty t^{p-3}\|f_{2,N_t}\|_{L^2_{2\pi}}^2dt
\end{equation}
It is well known that the Fourier transform is (up to a constant) an isometry from $L^2_{2\pi}$ to $l^2(\Z)$; we therefore have
$$\|f_{2,N_t}\|_{L^2_{2\pi}}^2=\sum_{|k|\geq N_t}|\widehat{f}(k)|^2.$$
Then by applying inequality  (\ref{eq2:ineg-sob}), we get
$$2\pi\|f\|_{L^p_{2\pi}}^p\leq 4p\sum_{k\in\Z^*}\int_{0}^\infty t^{p-3}{\bf 1}_{\{(t,k);\;|k|\geq N_t\}}(t,k)|\widehat{f}(k)|^2dt.$$
By using the definition of $N_t$, we have, for $k\in\Z^*$,
$$|k|\geq N_t\geq \frac{1}{2}\Big(\frac{t(1-2s)^{1/2}}{4C\|f\|_{H^s_{2\pi}}}\Big)^p\Longrightarrow t\leq \alpha_k=\frac{4C\|f\|_{H^s_{2\pi}}}{(1-2s)^{1/2}}2^{1/p}|k|^{1/p}.$$\\
By applying Fubini's theorem, we obtain($C=\sqrt{2}/\pi$)
$$\begin{array}{lll}
2\pi\|f\|_{L^p_{2\pi}}^p&\leq&4p\displaystyle\sum_{k\in\Z^*}\Big(\int_0^{\alpha_k}t^{p-3}dt\Big)|\widehat{f}(k)|^2\\
&\leq&4p\displaystyle\sum_{k\in\Z^*}\frac{{\alpha_k}^{p-2}}{p-2}|\widehat{f}(k)|^2\\
&\leq&4\displaystyle\frac{p}{p-2}2^{(p-2)/p}\big(\frac{4C}{(1-2s)^{1/2}}\big)^{p-2}\|f\|_{H^s_{2\pi}}^{p-2}\sum_{k\in\Z^*}|k|^{\frac{p-2}{p}}|\widehat{f}(k)|^2\\
&\leq&4\displaystyle\frac{p}{p-2}2^{(p-2)/p}\big(\frac{4C}{(1-2s)^{1/2}}\big)^{p-2}\|f\|_{H^s_{2\pi}}^{p}.
\end{array}$$
Finally, using $1-2s=\frac{2}{p}$ we get
$$\|f\|_{L^p_{2\pi}}^p\leq \frac{1}{2\pi}r_p^p\|f\|_{H^s_{2\pi}}^p,$$
with
$$r_p=\Big(4\frac{p}{p-2}2^{\frac{p-2}{p}}(4C)^{p-2}(\frac{p}{2})^{\frac{p-2}{2}}\Big)^{1/p}.$$
Using the elementary inequalities, for all $p\geq3$,
$$\frac{p}{p-2}\leq3\;\;{\rm and}\;\;2^{\frac{p-2}{p}}\leq2,$$
we get
$$r_p\leq R_p=\Big(24\times (4/\pi)^{p-2}p^{\frac{p-2}{2}}\Big)^{1/p}.$$
Then, we obtain
$$\|f\|_{L^p_{2\pi}}^p\leq \frac{1}{2\pi}R_p^p\|f\|_{H^s_{2\pi}}^p=\tilde{R}_p^p\|f\|_{H^s_{2\pi}}^p,$$
with $\tilde{R}_p=\Big(3\times (4/\pi)^{p-1}p^{\frac{p-2}{2}}\Big)^{1/p}.$
\item[$\bullet$] {\bf Second case:} Let $f\in C_{2\pi}^\infty$ be any element of $C_{2\pi}^\infty$.\\
Put $g=f-\widehat{f}(0)$. By the first case we obtain
$$\|g\|_{L^p_{2\pi}}^p\leq \tilde{R}_p^p\|g\|_{H^s_{2\pi}}^p.$$
Then, we get
$$\begin{array}{lcl}
\|f\|_{L^p_{2\pi}}&=&\|g+\widehat{f}(0)\|_{L^p_{2\pi}}\\
&\leq&\|g\|_{L^p_{2\pi}}+|\widehat{f}(0)|\\
&\leq&\tilde{R}_p\|g\|_{H^s_{2\pi}}+|\widehat{f}(0)|\\
&\leq&2\tilde{R}_p\|f\|_{H^s_{2\pi}}.
\end{array}$$
Which prove the required inequality with the constant $2\tilde{R}_p$.
\item[$\bullet$] {\bf Third case:} Let $f\in H^s_{2\pi}$.\\
Using the extension theorem for a continuous linear map defined on a dense subspace and the second case treated with $\overline{C^\infty_{2\pi}}=H_{2\pi}^s$, we obtain $H^s_{2\pi}\hookrightarrow L^p_{2\pi}$ and 
$$\|f\|_{L^p_{2\pi}}\leq 2\tilde{R}_p\|f\|_{H^s_{2\pi}},\;\forall f\in H^s_{2\pi}.$$
Which complete the proof of Proposition \ref{th:ineg-sob72}.
\end{enumerate}

\end{document}